\documentclass[]{interact}

\usepackage{epstopdf}% To incorporate .eps illustrations using PDFLaTeX, etc.
\usepackage[caption=false]{subfig}% Support for small, `sub' figures and tables

\usepackage[numbers,sort&compress]{natbib}% Citation support using natbib.sty
\bibpunct[, ]{[}{]}{,}{n}{,}{,}% Citation support using natbib.sty
\renewcommand\bibfont{\fontsize{10}{12}\selectfont}% Bibliography support using natbib.sty

\theoremstyle{plain}% Theorem-like structures provided by amsthm.sty
\newtheorem{theorem}{Theorem}[section]
\newtheorem{lemma}[theorem]{Lemma}
\newtheorem{corollary}[theorem]{Corollary}
\newtheorem{proposition}[theorem]{Proposition}

\theoremstyle{definition}

\newtheorem{example}[theorem]{Example}

\theoremstyle{remark}
\newtheorem{remark}{Remark}

\newcommand{\R}{\mathbb{R}}
\newcommand{\F}{F_{xy}^{(n)}}
\newcommand{\dol}{F_{xy}}
\newcommand{\hor}{F_{xy}^\star}
\newcommand{\di}{F_{xy}^{(n)}}
\newcommand{\dt}{\mathrm{d}t}
\newcommand{\zlk}{\left(\frac{1-p}{p}\right)^k}

\begin{document}

%\articletype{ARTICLE TEMPLATE}% Specify the article type or omit as appropriate

\title{Distributional Chaos in Random Dynamical Systems}

\author{
\name{Jozef Kov\'a\v c\textsuperscript{a}\thanks{CONTACT Jozef Kov\'a\v c. Email: jozef.kovac@fmph.uniba.sk} and Katar\'ina Jankov\'a\textsuperscript{a}}
\affil{\textsuperscript{a}Department of Applied Mathematics and Statistics, Faculty of Mathematics, Physics and Informatics, Comenius University in Bratislava, Mlynsk\'a dolina, Bratislava, Slovakia}
}

\maketitle

\begin{abstract}
In this paper, we introduce the notion of distributional chaos and the measure of chaos for random dynamical systems generated by two interval maps. We give some sufficient conditions for a zero measure of chaos and examples of chaotic systems. We demonstrate that the chaoticity of the functions that generate a system does not, in general, affect the chaoticity of the system, i.e., a chaotic system can arise from two nonchaotic functions and vice versa. Finally, we show that distributional chaos for random dynamical system is, in some sense, unstable.
\end{abstract}

\begin{keywords}
Random dynamical systems; iterated function systems; distributional chaos; measure of chaos
\end{keywords}

\section{Introduction}

In 1986, the Royal Society in London held an international conference on chaos. At this conference, the following informal definition of chaos was proposed:
\begin{center}
\emph{Stochastic behaviour occurring in a deterministic system.}\cite{Stewart}
\end{center}
Describing chaos mathematically can be very difficult and potentially ambiguous. However, there are many definitions that have attempted to capture the notion of chaos \cite[see e.g.][]{Ruette}.

The notion of chaos for discrete dynamical systems was first used in 1975 in a paper by Li and Yorke \cite{LiYorke}. They said that for a map $f$ defined on a closed interval $I$, the dynamical system 
\begin{equation}\label{system}
x_{n+1}=f(x_n)
\end{equation}
is chaotic if there exists an uncountable set $S\subset I$ such that for every pair of distinct points $x_0,y_0$ in this set, we have
\begin{center}
$
\begin{matrix}
\displaystyle\liminf_{n\to\infty}|x_n-y_n|=0 & \text{and} & \displaystyle\limsup_{n\to\infty}|x_n-y_n|>0.
\end{matrix}
$
\end{center} 
It was later shown that, for interval maps, the existence of one pair with such a property is sufficient for Li-Yorke chaos \cite{Kuchta}. 

A possible generalization of Li and Yorke's chaos is so called distributional chaos \cite[see][]{Smital1,Smital2}. For a map $f$ defined on a closed interval $I$ and points $x$ and $y$ in this interval, consider a real function $\di$ given by
\begin{equation}
F_{xy}^{(n)}(t)=\frac{1}{n}\#\{i\in\{0,1,\ldots,n-1\}: |f^i(x)-f^i(y)|<t\}, 
\end{equation}
where $f^0$ is the identity and $f^{n+1}\equiv f\circ f^n$. $\di(t)$ can be viewed as the probability that the distance between $x_J$ and $y_J$ is less than $t$, where $J$ is a uniformly randomly chosen time from the set $\{0,1,\ldots,n-1\}$. The system~(\ref{system}) is distributionally chaotic if this probability does not stabilize for some $x,y\in I$ and $t\in(a,b)\subseteq\R$, i.e., if 
\begin{equation}
\liminf_{n\to\infty} F_{xy}^{(n)}(t)<\limsup_{n\to\infty}F_{xy}^{(n)}(t).
\end{equation}
The function $\displaystyle\liminf_{n\to\infty} F_{xy}^{(n)}$ (resp. $\displaystyle\limsup_{n\to\infty} F_{xy}^{(n)}$) is called the lower (resp. upper) distribution function and is denoted by $F_{xy}$ (resp. $\displaystyle F_{xy}^\star$). A specific feature of distributional chaos is that, unlike many other types of chaos, it can be quantified by the so called (principal) measure of chaos, $\mu$. It is given by 
\begin{equation}
\mu(f)=\sup_{x,y\in I}\frac{1}{|I|}\int_0^{|I|}\left(\hor(t)-\dol(t)\right)\dt,
\end{equation}
which is the size of the area between the lower and the upper distribution function. 

This paper focuses on distributional chaos in the random dynamical system 
\begin{equation}\label{system0}
x_{n+1}=
\begin{cases}
f(x_n) & \text{ with probability } p,\\
g(x_n) & \text{ with probability } 1-p, 
\end{cases}
\end{equation}
where $p\in[0,1]$ and $f,g$ are functions defined on a closed interval $I$. The advantage of distributional chaos is its probabilistic interpretation, which enables us to easily redefine its notion for random dynamical systems. 

However, does it actually make sense to consider chaos in random dynamical systems - in which there is always some stochasticity?  The answer is, in fact, yes. The above-mentioned definitions are focused on the distances between two trajectories and these can have some `organized behaviour' - even in random dynamical systems. For example, if $I=[0,1]$, $f(x)=\frac{1}{2}x$, and $g(x)=\frac{1}{2}x+\frac{1}{2}$, then for any  $x_0$ and $y_0$, we always have (regardless of the function selection) $|x_1-y_1|=\frac{1}{2}|x_0-y_0|$, $|x_2-y_2|=\frac{1}{4}|x_0-y_0|$, $|x_2-y_2|=\frac{1}{8}|x_3-y_3|$, and so on. In this case, we observe a sort of different phenomenon - deterministic behaviour occurring in a random system. 

The system~(\ref{system0}) is also a so called iterated function system (IFS) with probabilities \cite[see][]{Barnsley}. As far as we know, the literature mostly focuses on the invariant measures in such systems \cite[e.g. ][]{Bhat,Diaconis,Stenflo}, and results concerning chaos are not common. In \cite{Kifer}, topological entropy was studied, and recently, some other chaotic notions in IFS were investigated in \cite{Bahabadi} and \cite{Ghane} (but randomness was not taken into account in these studies). 

This paper is organized as follows. In Section~2, we define the trajectory of the system~(\ref{system0}) (following \cite{Bhat} and \cite{Kifer}). In Section~3, we introduce distributional chaos and its measure for the system~(\ref{system0}). Section~4 focuses on some sufficient conditions for zero measure of chaos. In Section~5, we give two examples of distributionally chaotic systems. Section~6 deals with the  stability of distributional chaos. 

\section{Random dynamical system}
Let $\Omega$ denote the set of all sequences of the functions $f$ and $g$ ($\Omega=\{f,g\}^\infty$) and let $\mathcal{S}$ be the power set (the set of all subsets) of $\Omega$. $\mathcal{S}$ is trivially a $\sigma$-algebra on $\Omega$; hence, $(\Omega,\mathcal{S})$ is a measurable space. Let $P:\mathcal{S}\to[0,1]$ denote the probability measure on this space generated by the finite dimensional probabilities
\begin{equation}
P(\{\omega=(\omega_1,\omega_2,\ldots): \omega_{i_1}=\varphi_1,\ldots, \omega_{i_n}=\varphi_n\})=p^{\sum_{k=1}^n I_f(\varphi_k)}(1-p)^{\sum_{k=1}^n I_g(\varphi_k)},
\end{equation}
where $n$ and $i_1<i_2<\ldots<i_n$ are positive integers, $\varphi_1,\ldots,\varphi_n \in\{f,g\}$, and 
\begin{equation}\label{indikator}
I_f(\varphi)=
\begin{cases}
1 & \text{ if } \varphi=f \\
0 & \text{ if } \varphi\neq f.
\end{cases}
\end{equation}

$I_g(\cdot)$ is defined analogously. The trajectory of $x\in I$ of the random dynamical system~(\ref{system0}) can then be expressed as the stochastic process $\{x_n\}_{n=1}^\infty$ defined on $(\Omega,\mathcal{S},P)$ by
\begin{equation}
x_n(\omega)\equiv \omega_n(x_{n-1}(\omega)),
\end{equation}
where $x_0(\omega)\equiv x$. Or equivalently,
\begin{equation}
x_n(\omega)=\omega_n\circ\omega_{n-1}\circ\ldots\circ\omega_1(x).
\end{equation}
Given $x_n$, the random variable $x_{n+1}$ does not depend on $x_{n-1},x_{n-2},\ldots$. Therefore, $\{x_n\}_{n=1}^\infty$ is a Markov process. 

\section{Definition of distributional chaos and its measure}
Recall that in deterministic dynamical systems, the definition of distributional chaos is based on the function $\di(t)$. The value of $\di(t)$ can be viewed as the probability that the distance between $x_J$ and $y_J$ is less than $t$, where $J$ is a random variable with the uniform distribution on the set $\{0,1,\ldots,n-1\}$. Using this `probabilistic interpretation' in the random dynamical system~(\ref{system0}), we can define the function $\di(t,f,g,p)$ by 
\begin{flalign}
F_{xy}^{(n)}(t,f,g,p)&=P(|x_J-y_J|<t)=\nonumber \\
&=P(|x_J-y_J|<t\ |\ J=0)P(J=0)+\nonumber \\
&\ \ +P(|x_J-y_J|<t\ |\ J=1)P(J=1)+\ldots+\nonumber \\
&\ \ +P(|x_J-y_J|<t\ |\ J=n-1)P(J=n-1)=\nonumber \\
&=\frac{1}{n}P(|x_0-y_0|)+\ldots+\frac{1}{n}P(|x_{n-1}-y_{n-1}|)=\nonumber \\
&=\frac{1}{n}(P(|x_0-y_0|<t)+\ldots+P(|x_{n-1}-y_{n-1}|<t)).
\end{flalign}
$\di(t,f,g,p)$ defined in this way is also the expected value of
\begin{equation}
\frac{1}{n}\#\{i\in\{0,1,\ldots,n-1\}: |x_i-y_i|<t\}
\end{equation}
(this term is a random variable in random dynamical systems). To demonstrate this, we write
\begin{equation*}
\frac{1}{n}\#\{i\in \{0,\ldots n-1\}:|x_i-y_i|<t\}=
\frac{1}{n}\left(I_0(t,f,g,p)+\ldots+I_{n-1}(t,f,g,p)\right),
\end{equation*}
where 
\begin{equation}
I_k(t,f,g,p)=
\begin{cases} 
1 & \text{if } |x_k-y_k|<t,  \\ 
0 & \text{if } |x_k-y_k|\ge t
\end{cases}
\end{equation}
for $k=0,1,\ldots,n-1$. The random variable $I_k(t,f,g,p)$ has a Bernoulli distribution. Therefore, 
\begin{equation}
E(I_k(t,f,g,p))=P(|x_k-y_k|<t).
\end{equation} 
Hence, 
\begin{align}
E\left(\frac{1}{n}\#\{i\in \{0,\ldots n-1\}:|x_i-y_i|<t\}\right)&=\nonumber \\
\frac{1}{n}(P(|x_0-y_0|<t)&+\ldots+P(|x_{n-1}-y_{n-1}|<t)).
\end{align} 
Given $\di(t,f,g,p)$, the lower and upper distribution functions can be defined in the same way as in the deterministic system, i.e., 
\begin{eqnarray}
\dol(t,f,g,p) &=& \liminf_{n\to\infty} \di(t,f,g,p),\\
\hor(t,f,g,p) &=& \limsup_{n\to\infty} \di(t,f,g,p).
\end{eqnarray}
We can also define the measure of chaos as 
\begin{equation}
\mu(f,g,p)=\sup_{x,y\in I}\frac{1}{|I|}\int_0^{|I|}(F_{xy}^\star(t,f,g,p)-F_{xy}(t,f,g,p))\mathrm{d}t.
\end{equation}
Given that for the functions $\dol(t,f,g,p)$ and $\hor(t,f,g,p)$ we clearly have 
\begin{equation}
0\le \dol(t,f,g,p)\le\hor(t,f,g,p)\le 1,
\end{equation} 
the measure $\mu(f,g,p)$ is always nonnegative. If the measure is positive, we say that the system~(\ref{system0}) is distributionally chaotic. 

\begin{remark}
For simplicity of notation, we use the same notation as in the deterministic system in the next sections, i.e., $\di(t), \dol(t)$, and $\hor(t)$. We also omit $p$ from $\mu(f,g,p)$.
\end{remark}

\section{Zero measure of distributional chaos}

In general, it can be very difficult to calculate the measure $\mu(f,g)$. However, in some cases, we are able show that this measure is zero. 

\begin{lemma}
If for every $t>0$ and every $x,y\in I$, the limit 
\begin{equation}
\lim_{n\to\infty} \F(t) = \lim_{n\to\infty} 
  \frac{P(|x_0-y_0|<t)+\ldots+P(|x_{n-1}-y_{n-1}|<t)}{n}
\end{equation}
exists, then $\mu(f,g)=0$.
\end{lemma}

\begin{proof}
Directly from the definition.
\end{proof}

%%%%%%%%%%%%%%%%%%%%%%%%%%%%%%%%%%%%%%%%%%%%%%%%%%%%%%

\begin{corollary}\label{post1}
If for every $t>0$ and every $x,y\in I$, the limit
\begin{equation}
\lim_{n\to\infty} P(|x_n-y_n|<t)
\end{equation}
exists, then $\mu(f,g)=0$.
\end{corollary}

\begin{proof}
Directly from the fact that the sequence of arithmetic means of a convergent sequence also converges.
\end{proof}

%%%%%%%%%%%%%%%%%%%%%%%%%%%%%%%%%%%%%%%%%%%%%%%%%%%%%%

\begin{proposition}
If the functions $f$ and $g$ are contractive, then $\mu(f,g)=0$.
\end{proposition}

\begin{proof}
Take arbitrary $x,y\in I$. Given that $f$ and $g$ are contractive, there exist $c_1$, $c_2\in (0,1)$ such that $|f(x)-f(y)|\le c_1|x-y|$ and $|g(x)-g(y)|\le c_2|x-y|$ for every $x,y\in I$. Denote $c\equiv \max(c_1,c_2)$. Clearly, for every $\omega = (\omega_1,\omega_2,\ldots)\in \Omega$, $|x_1-y_1|=|\omega_1(x)-\omega_1(y)|\le c|x-y|$, and similarly, $|x_n-y_n|\le c^n|x-y|$. As $c<1$, this term tends to zero as $n$ tends to infinity. Consequently, $P(|x_n-y_n|<t)\to 1$ for every $t>0$. Hence, by Corollary~\ref{post1}, we have $\mu(f,g)=0$.
\end{proof}

%%%%%%%%%%%%%%%%%%%%%%%%%%%%%%%%%%%%%%%%%%%%%%%%%%%%%%%

\begin{theorem}\label{Lips}
Let the function $f$ be Lipschitz continuous (i.e. there is $M<\infty$ such that $|f(x)-f(y)|\le M|x-y|$ for every $x,y\in I$) and the function $g$ be contractive (i.e. there is $c<1$ such that $|g(x)-g(y)|\le c|x-y|$ for every $x,y\in I$). Next, suppose that $cM\ge 1$. Let $r$ be the smallest positive integer for which $c^rM\le 1$. If the probability $p$ of choosing the function $f$ in the $n$-th step is smaller than $\frac{1}{1+r}$, then $\mu(f,g)=0$.
\end{theorem}

Before proving Theorem~\ref{Lips}, we need the following lemma. 

\begin{lemma}\label{binom}
Let $X_n$ have a binomial distribution with parameters $n$ and $p\in(0,1)$, and let $a$ and $b$ be arbitrary real numbers. Then 
\begin{enumerate}
\item[(i)] if $a>p$, $\lim_{n\to\infty} P(X_n\ge an+b)=0$ and
\item[(ii)] if $0<a<p$, then $\lim_{n\to\infty} P(X_n\ge an+b)=1$.
\end{enumerate}
\end{lemma}

\begin{proof}
Recall that the expected value of $X_n$ is $np$ and the variance of $X_n$ is $np(1-p)$. Using Chebyshev's inequality, we have 
\begin{enumerate}
\item[(i)]
$
\begin{aligned}[t]
P(X_n\ge an+b)&=P(X_n-np\ge (a-p)n+b)\le \\ 
&\le P(|X_n-np|\ge (a-p)n+b)\le \frac{np}{((a-p)n+b)^2} \to 0
\end{aligned}
$\\
as $n\to\infty$, and
\item[(ii)]
$
\begin{aligned}[t]
P(X_n< an+b)&=P(X_n-np < (a-p)n+b) = \\ 
& = P(np-X_n > (p-a)n-b)\le \\
& \le P(|np-X_n|>(p-a)n-b)\le \\
& \le P(|X_n-np|\ge (p-a)n-b)\le \frac{np}{((p-a)n-b)^2}\to 0,
\end{aligned}
$\\
as $n\to \infty$; hence, $P(X_n\ge an+b)\to 1$.
\end{enumerate}
\end{proof}

\begin{proof}[Proof of Theorem~\ref{Lips}]
Take arbitrary $x,y\in I$ and $t>0$. For fixed $t$, there is an integer $k$ such that $c^k|I|<t$; hence, for any positive integer $m$, 
\begin{equation}
c^k(c^rM)^m|I|<t.
\end{equation}
Let $X_n$ be a random variable representing the number of times we applied the function $f$ in the first $n$ steps. More formally, for $\omega=(\omega_1,\omega_2,\ldots)\in \Omega$, we set
\begin{equation}
X_n(\omega)=\sum_{i=1}^n I_f(\omega_i),
\end{equation}
where $I_f$ was defined in (\ref{indikator}). Next, the properties of the functions $f$ and $g$ imply that 
\begin{equation}
|x_n-y_n|\le c^{n-X_n}M^{X_n}|x-y|\le c^{n-X_n}M^{X_n}|I|.
\end{equation}
Hence, if $n-X_n\ge rX_n+k$, then $|x_n-y_n|<t$. In the worst case ($n-X_n=rX_n+k$), we have
\begin{equation}
|x_n-y_n|\le c^{rX_n+k}M^{X_n}|I|=c^k(c^r)^{X_n}M^{X_n}|I|=c^k(c^rM)^{X_n}|I|<t.
\end{equation} 
Clearly, $X$ has a binomial distribution with parameters $n$ and $p$, and using Lemma~\ref{binom}, we obtain
\begin{eqnarray*}
P(|x_n-y_n|<t)&\ge& P(n-X_n \ge rX_n+k) = P(rX_n+X_n\le n-k) =\\
&=& P\left(X_n\le \frac{1}{1+r}n-\frac{k}{1+r}\right) 
= 1-P\left(X_n>\frac{1}{1+r}n-\frac{k}{1+r}\right)\ge \\
&\ge& 1-P\left(X_n\ge\frac{1}{1+r}n-\frac{k}{1+r}\right) \to 1,
\end{eqnarray*}
because $\frac{1}{1+r}>p$. Therefore, $P(|x_n-y_n|<t)\to 1$ for every $x,y\in I$ and every $t>0$, and by Corollary~\ref{post1}, $\mu(f,g)=0$.
\end{proof}

%%%%%%%%%%%%%%%%%%%%%%%%%%%%%%%%%%%%%%%%%%%%%%%%%%%%%%%%%%%

\begin{theorem}
Let the function $f$ be Lipschitz continuous, the function $g$ be contractive (with the same constants $M$ and $c$ as in the Theorem~(\ref{Lips}) and let $cM\le 1$. Let $r$ be the greatest positive integer for which $cM^r\le 1$. If the probability $p$ is smaller than $\frac{r}{1+r}$, then $\mu(f,g)=0$. 
\end{theorem}

The proof is analogous to the previous one. 

\begin{theorem}\label{Finite}
Let $x_n$ converge to a finite set $A\equiv\{a_1,\ldots,a_m\}$ for any $x\in I$ in such sense that
\begin{equation}
\lim_{n\to\infty} P(x_n\in A)=1.
\end{equation}
Then $\mu(f,g)=0$.
\end{theorem}

For the sake of simplicity, we will first formulate and prove two lemmas.

\begin{lemma}\label{MR1}
Let $\{Z_n\}_{n=0}^\infty$ be a Markov chain with a finite state space $S$. Then for every $i,j\in S$, the limit
\begin{equation}\label{limita}
\lim_{n\to\infty}\frac{1}{n}\sum_{k=0}^n P(Z_k=j|Z_0=i)
\end{equation}
exists.
\end{lemma}

\begin{proof}
If this chain is irreducible, then the limit exists and converges to a unique stationary distribution (see e.g. \cite{Bhat}). Now, let this chain be reducible, so that $S=C_1\cup\ldots\cup C_s\cup T$, where $C_1,\ldots, C_s$ are closed subsets of $S$ such that the chain restricted to $C_l$, $l=1,\ldots,s$ is irreducible and $T$ is the set of all transient states. Let $\Pi(C_l)$ denote the unique stationary distribution on $C_l$, $l=1,\ldots,s$. If
\begin{enumerate}
\item $i\in C_l$, then
\begin{itemize}
\item if $j\in C_l$, then the limit in~(\ref{limita}) is equal to $\Pi_j(C_l)$,
\item if $j\notin C_l$, then the limit in~(\ref{limita}) is clearly equal to 0,
\end{itemize}
\item $i\in T$, i.e.,  $i$ is transient, then
\begin{itemize}
\item if $j$ is also transient, then the limit in~(\ref{limita}) is clearly 0,
\item if $j\in C_l$, then the limit in~(\ref{limita}) is equal to $\Pi_j(C_l)\cdot\nu(C_l)$, where $\nu(C_l)$ is the conditional probability of hitting the set $C_l$ if $Z_0=i$.   
\end{itemize}
\end{enumerate}
This list covers all possibilities. 
\end{proof}

\begin{lemma}\label{MR2}
If $x,y\in A$, then for any $t>0$, the limit
\begin{equation}\label{limita2}
\lim_{n\to\infty}\frac{1}{n}\sum_{i=0}^{n-1} P(|x_i-y_i|<t)
\end{equation}
exists.
\end{lemma}

\begin{proof}
Consider a Markov chain $\{Z_n\}_{n=0}^\infty$ with states $d_{ij}$, where $i,j=1,\ldots,m$, given in such way that $Z_n$ is in the state $d_{ij}$ if and only if $x_n=a_i$ and $y_n=a_j$. Without loss of generality, suppose that $Z_0=d_{12}$. Let $K$ be the set given by 
\begin{equation}
K\equiv \{(i,j)\in\{1,\ldots,m\}^2:|a_i-a_j|<t\}.
\end{equation}
It can be seen that $|x_n-y_n|<t$ if and only if
\begin{equation}
Z_n\in\{d_{ij}, (i,j)\in K\}.
\end{equation}
It follows that 
\begin{eqnarray*}
\lim_{n\to\infty}\frac{1}{n}\sum_{k=0}^{n-1} P(|x_k-y_k|<t)&=&
\lim_{n\to\infty}\frac{1}{n}\sum_{k=0}^{n-1} \sum_{(i,j)\in K} P(Z_k = d_{ij}|Z_0=d_{12})=\\
&=&\sum_{(i,j)\in K} \lim_{n\to\infty}\frac{1}{n}\sum_{k=0}^{n-1} P(Z_i = d_{ij}|Z_0=d_{12}),
\end{eqnarray*}
which is a finite sum of existing limits (from Lemma~\ref{MR1}). Hence, the limit in~(\ref{limita2}) exists.
\end{proof}

\begin{proof}[Proof of Theorem~\ref{Finite}]
Let $x,y\in I$, $t>0$, and $\delta>0$ be arbitrary. We will show that $\hor(t)-\dol(t)\le \delta$. 

From our assumptions, for the given $\delta$, there exists $n_0$ such that
\begin{equation}\label{ineq}
P(x_{n_0}\in A \wedge y_{n_0}\in A)>1-\delta.
\end{equation}
By the time of $n_0$, there are only finitely many ($2^{n_0}$) possible scenarios (e.g., if $n_0=3$, then the possible scenarios are $fff,ffg,fgf,gff,fgg,gfg,ggf,$ and $ggg$). These scenarios can be expressed by the sets 
\begin{equation}
B_{\varphi_1\varphi_2\ldots\varphi_{n_0}}\equiv\{\omega=(\omega_1,\omega_2,\ldots)\in\Omega:\omega_1=\varphi_1,\omega_2=\varphi_2,\ldots,\omega_{n_0}=\varphi_{n_0}\},
\end{equation}
where $\varphi_1,\ldots,\varphi_{n_0}\in\{f,g\}$.

For notational simplicity, we denote these sets by $C_1,\ldots,C_{2^{n_0}}$ and sort them so that
\begin{itemize}
\item if $\omega\in C_1\cup\ldots\cup C_k$, then $x_n(\omega)\notin A$ or $y_n(\omega)\notin A$ and 
\item if $\omega\in C_{k+1}\cup\ldots\cup C_{2^{n}}$, then $x_n(\omega)\in A$ and $y_n(\omega)\in A$.
\end{itemize}
The sets $C_1,\ldots,C_k$ are clearly disjoint; hence, 
\begin{equation}
P(C_1)+\ldots+P(C_k)\le\delta.
\end{equation}
(from~(\ref{ineq})). Now we have 
\begin{equation*}
\begin{aligned}
&\hor(t)-\dol(t)=\\
&=\limsup_{n\to\infty}\frac{1}{n}\sum_{i=0}^{n-1}P(|x_i-y_i|<t)-
\liminf_{n\to\infty}\frac{1}{n}\sum_{i=0}^{n-1}P(|x_i-y_i|<t)=\\
&=
\limsup_{n\to\infty}\frac{1}{n}\sum_{i=0}^{n-1}\sum_{j=1}^{2^{n_0}}P(|x_i-y_i|<t|C_j)P(C_j)-\\
&\phantom{=}-
\liminf_{n\to\infty}\frac{1}{n}\sum_{i=0}^{n-1}\sum_{j=1}^{2^{n_0}}P(|x_i-y_i|<t|C_j)P(C_j)\le\\
&
\sum_{j=1}^{2^{n_0}}P(C_j)\left(
\limsup_{n\to\infty}\frac{1}{n}\sum_{i=0}^{n-1}P(|x_i-y_i|<t|C_j)-
\liminf_{n\to\infty}\frac{1}{n}\sum_{i=0}^{n-1}P(|x_i-y_i|<t|C_j)
\right).
\end{aligned}
\end{equation*}
However, $x_{n_0}(\omega)$ and $y_{n_0}(\omega)$ are already in the set $A$ if $\omega\in C_j, j=k+1\ldots,2^{n_0}$. Therefore, the superior limit and inferior limit are equal (by Lemma~\ref{MR2}). This follows from the fact that if $C_j$ is of the form $B_{\varphi_1\varphi_2\ldots\varphi_{n_0}}$, then 
\begin{equation}
\limsup_{n\to\infty}\frac{1}{n}\sum_{i=0}^{n-1}P(|x_i-y_i|<t|C_j)
\end{equation}
is equal to $\displaystyle F_{x'y'}^\star(t)$, where $x'=\varphi_{n_0}\circ\ldots\circ\varphi_1(x)$, and a similar argument holds for $y'$. Consequently,
\begin{equation*}
\begin{aligned}
&\hor(t)-\dol(t)\le\\
&
\sum_{j=1}^{k}P(C_j)\left(
\limsup_{n\to\infty}\frac{1}{n}\sum_{i=0}^{n-1}P(|x_i-y_i|<t|C_j)-
\liminf_{n\to\infty}\frac{1}{n}\sum_{i=0}^{n-1}P(|x_i-y_i|<t|C_j)
\right)\le\\
&\le
\sum_{j=1}^kP(C_j)\le\delta.
\end{aligned}
\end{equation*}
Since $\delta>0$ was arbitrary, $\hor(t)-\dol(t)=0$ for every $t>0$. Therefore, $\mu(f,g)=0$.
\end{proof}

\section{Examples of distributionally chaotic systems}
In this section, we will give two examples of distributionally chaotic systems and calculate their measure of chaos. In the first example, the measure is a continuous function of $p$. In the second example, the measure of chaos is constant for every $p\in(0,1)$.   

\begin{example}\label{pr1}
Consider the functions $f,g:[0,1]\to[0,1]$, where 
\begin{equation}
f(x)=
\begin{cases}
3x & \text{ if } x  \in [0,\frac{1}{3}],\\
-3x+2 & \text{ if } x \in (\frac{1}{3},\frac{2}{3}),\\
3x-2 & \text{ if } x \in [\frac{2}{3},1] 
\end{cases}
\end{equation}
and $g(x)=\frac{1}{3}x$. We will show that the measure of chaos of the system generated by these two functions is 
\begin{equation}
\mu(f,g)=
\begin{cases}
0 & \text{ if } p< \frac{1}{2},\\
\frac{6p-3}{4p-1} & \text{ if } p\ge\frac{1}{2}.
\end{cases}
\end{equation}
The case where $p<\frac{1}{2}$ simply follows from Theorem~\ref{Lips}. For the case $p\ge\frac{1}{2}$, we will prove the following propositions. 

\begin{figure}[ht]
\begin{center}
\includegraphics[height=8cm,width=8cm]{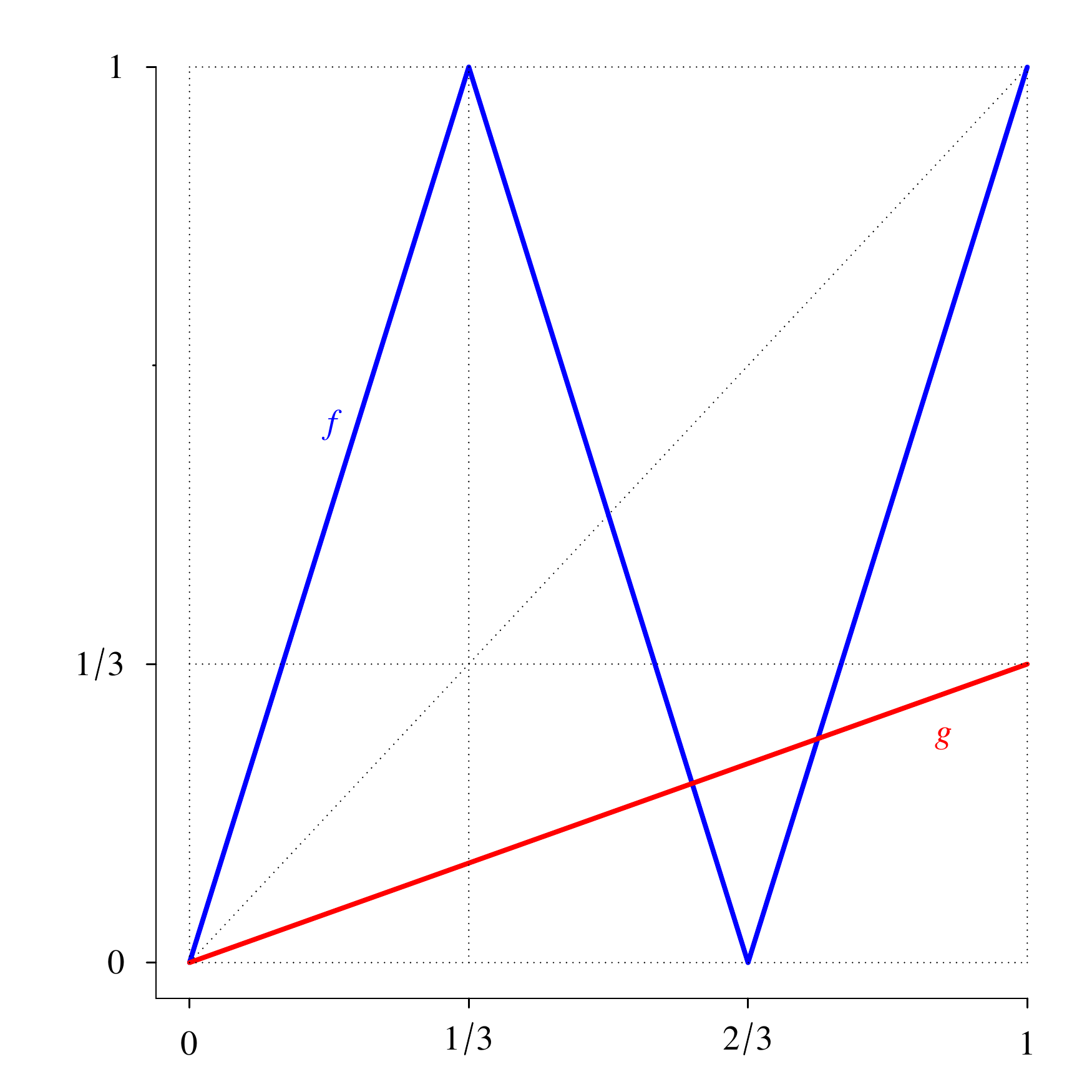}
\end{center}
\caption{The functions $f$ and $g$ in Example~\ref{pr1}}
\end{figure}

\begin{proposition}\label{dva}
If $p>\frac{1}{2}$, then for every $x,y\in [0,1]$, 
\begin{equation}
\int_0^1 \hor(t)-\dol(t)\dt \le \frac{6p-3}{4p-1}.
\end{equation}
\end{proposition}

\begin{proposition}\label{tri} If $p>\frac{1}{2}$, then for $y\equiv 0$, there exists a sequence $\{x^{(k)}\}_{k=1}^\infty$ of points in $[0,1]$ such that 
\begin{equation}
\lim_{k\to\infty}\int_0^1 F_{x^{(k)}y}^\star(t)-F_{x^{(k)}y}(t)\dt = \frac{6p-3}{4p-1}.
\end{equation}
\end{proposition}
In the next part, we will use the ternary representation of the numbers in [0,1]. We begin with some notation:
\begin{itemize}
\item $s$ will denote any infinite sequence of zeros, ones, or twos;
\item $r$ will denote any finite sequence of zeros, ones, or twos; $\ell(r)$ will denote the length of the sequence $r$;
\item for a positive integer $k$ and for $k=\infty$, $0^k$, $1^k$, and $2^k$ will denote the sequence of zeros, ones, and twos of length $k$, respectively. $0^0$, $1^0$, and $2^0$ will denote an `empty symbol' (e.g., $r0^0s=rs$); 
\item given that every number in [0,1] can be written as $0.s$ for some $s$, we omit `$0.$';
\item for a positive integer $k$, $s(k)$ will denote the $k$-th term of the sequence $s$;
\item for $x,y\in [0,1]$, we define 
\begin{equation}
U(x,y)\equiv\min\{k\in\{1,2,\ldots\},\text{ such that } s_x(k)\neq s_y(k)\}-1,
\end{equation}
where $s_x$ and $s_y$ are the ternary representations of $x$ and $y$, respectively. If there is ambiguity (e.g., $s_x=r02^\infty=r10^\infty$), $s_x$ and $s_y$ are chosen such that $U(x,y)$ is maximal.
For example, $U(\frac{1}{3},\frac{2}{9})=2$ because $s_\frac{1}{3}=02222\ldots$ and $s_\frac{2}{9}=02000\ldots$. Note that if $U(x,y)\ge k$, then $|x-y|\le 3^{-k}$.
\end{itemize}
Next, from the definition of the functions $f$ and $g$, it can be seen that
\begin{itemize}
\item $g(s)=0s$,
\item $f(0s)=f(2s)=s$,
\item $f(1s)=\overline{s}$, where $\overline{s}$ is the sequence obtained from $s$ by interchanging zeros and twos in each place,
\item $f(g(s))=s$, and
\item $f(0^\infty)=g(0^\infty)=0^\infty$.
\end{itemize}
In order to prove Proposition~\ref{dva}, we require the following lemma.
\begin{lemma}\label{rw}
Let $Z_1, Z_2, \ldots$ be independent and identically distributed random variables, where $P(Z_1=1)=p$ and $P(Z_1=-1)=1-p$ for $p\in [0,1]$. Next, consider a simple random walk $\{S_n\}_{n=0}^\infty$, where $S_0=0$ and $S_n=\sum_{i=1}^n Z_i$.
\begin{itemize}
\item If $p < \frac{1}{2}$, then $P(S_n=-k \text{ for some } n\in\{1,2,\ldots\})=1$, and
\item if $p \ge \frac{1}{2}$, then $\displaystyle P(S_n=-k \text{ for some } n\in\{1,2,\ldots\})=\left(\frac{1-p}{p}\right)^k$, 
\end{itemize}
where $k$ is an arbitrary positive integer. 
\end{lemma}
A proof of this lemma can be found in \cite{Feller}.
%%%%%%%%%%%%%%%%%
\begin{proof}[Proof of Proposition~\ref{dva}]
Let $x,y\in [0,1]$ be arbitrary. Let $A_1^n$ denote the set 
\begin{equation}
A_1^n=\{(\omega_1,\omega_2,\ldots)\in\Omega: \omega_n=g\}.
\end{equation}
As $g(s)=0s$, we have $U(x_n(\omega),y_n(\omega))\ge 1$ for every $\omega\in A_1^n$ because, in this case, the ternary representations of $x_n=g(x_{n-1})$ and $y_n=g(y_{n-1})$ begin with zero. Next, consider the set
\begin{equation}
A_3^n=\left\{\omega\in\Omega: \sum_{i=0}^2(I_f(\omega_{n-i})-I_g(\omega_{n-i}))=-1\right\},
\end{equation}
where $I_f$ and $I_g$ are as defined in~(\ref{indikator}). If $\omega\in A_3^n$, then, from the definition of the set, there are two $g$'s and one $f$ among $\omega_{n-2}$, $\omega_{n-1}$, and $\omega_{n}$. The order of these functions is either $f,g,g$ (then $x_n=g\circ g\circ f(x_{n-3})$ begins with zero because $g(s)=0s$), or there is a $g$ followed by $f$ (which is the identity), and $x_n=\omega_n\circ\omega_{n-1}\circ\omega_{n-2}(x_{n-3})=g(x_{n-3})$ begins with zero. The same is true for $y_n$. Hence, $U(x_n(\omega),y_n(\omega))\ge 1$ for every $\omega\in A_3^n$.
Similarly, we can define the set 
\begin{equation}
A_m^n=
\left\{
\omega\in\Omega: \sum_{i=0}^{m-1}(I_f(\omega_{n-i})-I_g(\omega_{n-1}))=-1
\right\},
\end{equation}
where $m=1,2,\ldots,n$. If $m$ is even, then $A_m^n$ is clearly empty. If $m$ is odd and $\omega\in A_{m}^n$, then there are $\frac{m+1}{2}$ $g$'s and $\frac{m-1}{2}$ $f$'s among $\omega_{n-m+1}, \omega_{n-m+2},\ldots,\omega_n$. Again, there are two possibilities: 
\begin{enumerate}
\item the order of these functions is $f,f,\ldots,f,g,g,\ldots,g$, and 
\begin{equation}
x_n=g\circ g\circ\ldots\circ g\circ f\circ\ldots\circ f(x_{n-m}) = g(x_{n-1})
\end{equation}
begins with zero, or
\item there is an $f\circ g$. Without loss of generality, assume that $\omega_{n-2}=f$ and $\omega_{n-3}=g$; then 
\begin{equation}
\omega_n\circ\omega_{n-1}\circ f\circ g\circ \omega_{n-4}\circ\ldots\circ\omega_{n-m+1}=\omega_n\circ\omega_{n-1}\circ \omega_{n-4}\circ\ldots\circ\omega_{n-m+1}
\end{equation}
because $f\circ g$ is the identity. Again, the order of the remaining functions $\omega_{n-m+1},\ldots, \omega_{n-4}, \omega_{n-1},\omega_n$ is either $f,\ldots,f,g\ldots,g$, or there is a $f\circ g$, which can be `removed'. This can be repeated until we get $x_n=g(x_k)$ for some $k$, which begins with zero. 
\end{enumerate}
Therefore, if $\omega\in A_m^n$, then $U(x_n(\omega),y_n(\omega))\ge 1$.

Now, for $t>\frac{1}{3}$, we have
\begin{equation}
P(|x_n-y_n|<t)\ge P(U(x_n,y_n)\ge 1) \ge P(B_n^1),
\end{equation}
where $\displaystyle B_n^1=\bigcup_{i=1}^n A_i^n$, which can also be written as 
\begin{equation}
B_n^1=
\left\{
\omega\in\Omega: \sum_{i=0}^{m-1}(I_f(\omega_{n-i})-I_g(\omega_{n-i}))=-1 \text{ for some } m\in\{1,2,\ldots,n\}
\right\}.
\end{equation}
However, 
\begin{equation}
I_f(\omega_{n-i})-I_g(\omega_{n-i})=
\begin{cases}
1 & \text{ if } \omega_{n-i}=f \text{ (with probability } p),\\
-1 & \text{ if } \omega_{n-i}=g \text{ (with probability } 1-p).
\end{cases}
\end{equation}
Hence, the sum 
\begin{equation}
\sum_{i=0}^{m-1}(I_f(\omega_{n-i})-I_g(\omega_{n-i}))
\end{equation}
is a simple random walk. Using Lemma~\ref{rw}, we get 
\begin{equation}
\lim_{n\to\infty} P(B_n^1)=\frac{1-p}{p}
\end{equation}
because $p\ge \frac{1}{2}$. Similarly, we can construct the set 
\begin{equation}\label{Bn}
B_n^k\equiv
\left\{
\omega\in\Omega: \sum_{i=0}^{m-1}(I_g(\omega_{n-i})-I_f(\omega_{n-i}))=-k \text{ for some } m\in\{1,2,\ldots,n\}
\right\},
\end{equation}
where $k=1,2,\ldots$. Using the same arguments as above, we have
\begin{itemize}
\item $\displaystyle \lim_{n\to\infty} P(B_n^k)=\left(\frac{1-p}{p}\right)^k$,
\item if $\omega\in B_n^k$, then $U(x_n(\omega),y_n(\omega))\ge k$; hence, $|x_n(\omega)-y_n(\omega)|\le 3^{-k}$.
\end{itemize}
Now let $t\in (0,1]$ be arbitrary and let $k$ be an integer for which $t\in (3^{-k},3^{-k+1}]$. We have
\begin{equation}
\liminf_{n\to\infty} P(|x_n-y_n|<t)\ge \liminf_{n\to\infty} P(U(x_n,y_n)\ge k)\ge \liminf_{n\to\infty} P(B_n^k) = \left(\frac{1-p}{p}\right)^k. 
\end{equation}
It follows that 
\begin{equation}
F_{xy}(t)=\liminf_{n\to\infty}\frac{1}{n}\sum_{i=0}^{n-1} P(|x_i-y_i|<t) \ge \left(\frac{1-p}{p}\right)^k
\end{equation}
for $t\in (3^{-k},3^{-k+1}]$, where $k=1,2,\ldots$. Because $\displaystyle\hor(t)$ is always lower than $1$, the maximal possible area between $\dol$ and $\displaystyle\hor$ is  
\begin{equation*}
\begin{aligned}[t]
\int_0^1 \hor(t)&-\dol(t)\dt \le \int_0^1 1 - \dol(t)\dt = 1-\int_0^1 \dol(t)\dt \le \\ 
& \le 1-\sum_{k=1}^\infty (3^{-k+1}-3^{-k})\zlk = 1- 2\sum_{k=1}^\infty \left(\frac{1}{3}\right)^k\zlk = \\
& = 1-2\sum_{k=1}^\infty \left(\frac{1-p}{3p}\right)^k =1- 2\frac{\frac{1-p}{3p}}{1-\frac{1-p}{3p}}=1- \frac{2-2p}{4p-1} =\frac{6p-3}{4p-1}.
\end{aligned}
\end{equation*}
\end{proof}
\begin{figure}
\begin{center}
\includegraphics[height=8cm,width=8cm]{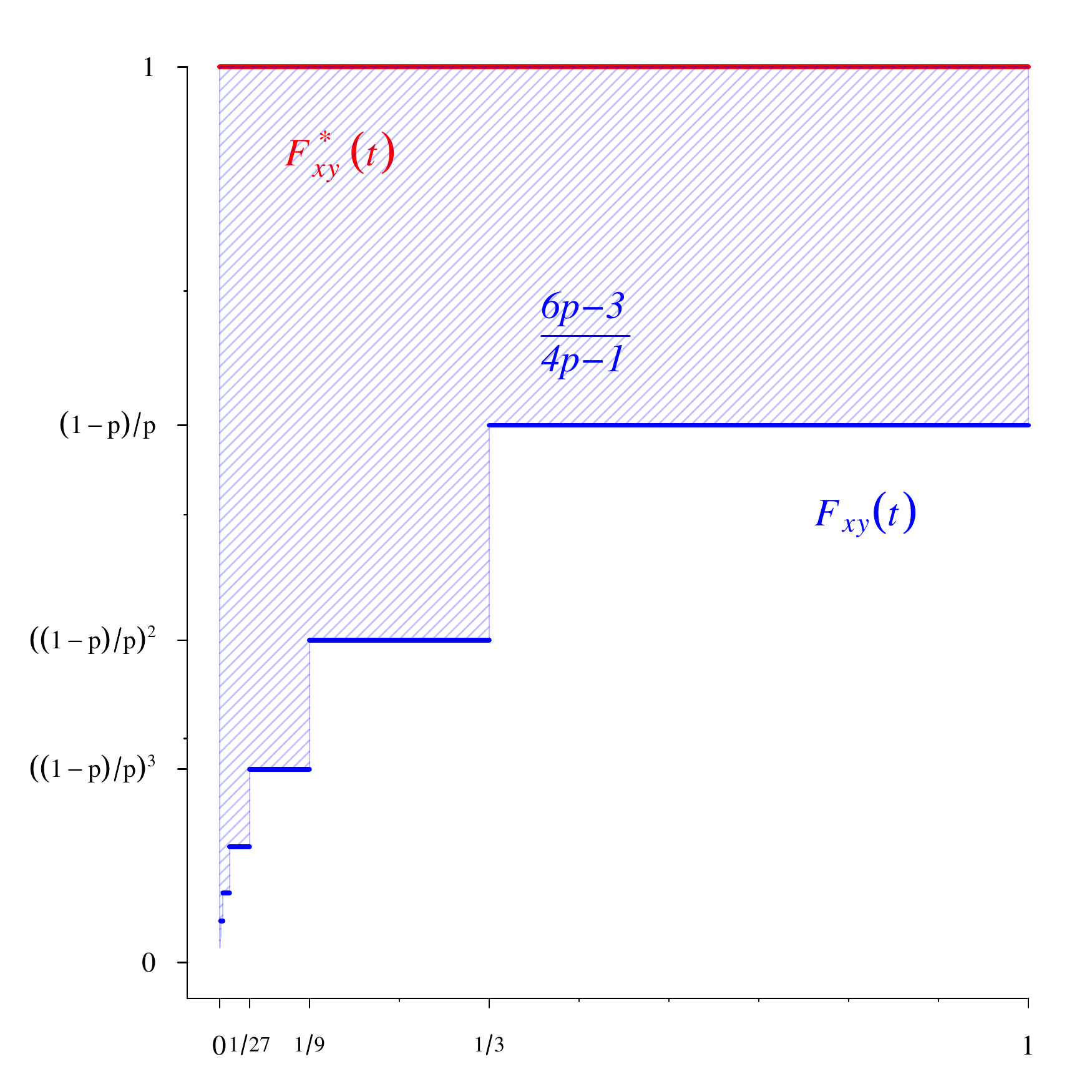}
\end{center}
\caption{Maximal possible area between $F_{xy}$ and $F_{xy}^\star$.}
\end{figure}
Before proving Proposition~\ref{tri}, we state two technical lemmas.
\begin{lemma}\label{pat}
Let $t>0$ and $\varepsilon>0$ be arbitrary. For every finite sequence $r$ that does not contain ones and every infinite sequence $s$, there exist positive integers $M$ and $N$ such that for $y=0$ and the number $w$ with the ternary representation $r0^Ms$, we have 
\begin{equation}
|F_{wy}^{(N)}(t)-1|<\varepsilon.
\end{equation}
Moreover, $N$ can be chosen to be arbitrarily large.
\end{lemma}
\begin{lemma}\label{sest}
Let $\varepsilon>0$ and $t_l^k=3^{-l}-3^{-k}$, where $k$ is an arbitrary (but fixed) positive integer and $l\in\{0,1,\ldots,k-1\}$. For every finite sequence $r$ that does not contain ones and every infinite sequence $s\neq 2^\infty$ there exist positive integers $M$ and $N$ such that for $y=0$ and number $z$ with the ternary representation $r2^Ms$, we have 
\begin{equation}
\left|F_{zy}^{(N)}(t_l^k)-\left(\frac{1-p}{p}\right)^{l+1}\right|<\varepsilon
\end{equation}
for every $l\in\{0,1,\ldots,k-1\}$.
Moreover, $N$ can be chosen to be arbitrarily large.
\end{lemma}

We can now prove Proposition~\ref{tri} (we will prove the lemmas later).
\begin{proof}[Proof of Proposition~\ref{tri}]
First, consider only $t_0^1\equiv\frac{2}{3}=3^0-3^{-1}$ and proceed as follows:
\begin{itemize}
\item choose any $r_1$ that does not contain ones or $s$;
\item for $\varepsilon_1\equiv 1$, there exist $M_1$ and $N_1$ such that for $u^1\equiv r_10^{M_1}s$, we have 
\begin{equation*}
\left|F_{u^1y}^{(N_1)}(t_0^1)-1\right|<1
\end{equation*}
(Lemma~\ref{pat});
\item for $\varepsilon_2\equiv\frac{1}{2}$, $r_2=r_10^{M_1}$, and $s$, there exist $M_2$ and $N_2>N_1$ such that for $u^2\equiv r_22^{M_2}s$, we have
\begin{equation*}
\left|F_{u^2y}^{(N_2)}(t_0^1)-\frac{1-p}{p}\right|<\frac{1}{2}
\end{equation*}
(Lemma~\ref{sest});
\item for $\varepsilon_3\equiv\frac{1}{3}$, $r_3=r_10^{M_1}2^{M_2}$, and $s$ there exist $M_3$ and $N_3>N_2$ such that for $u^3\equiv r_30^{M_3}s$, we have
\begin{equation*}
\left|F_{u^3y}^{(N_3)}(t_0^1)-1\right|<\frac{1}{3};
\end{equation*}
\item $\ldots$.
\end{itemize}
From this construction, it can be seen that for $x^{(1)}\equiv r_10^{M_1}2^{M_2}0^{M_3}2^{M_4}\ldots$, we have 
\begin{eqnarray}
F_{x^{(1)}y}^\star(t_0^1)&=&\limsup_{n\to\infty} F_{x^{(1)}y}^{(n)}(t_0^1)=1,\\
F_{x^{(1)}y}(t_0^1)&=&\liminf_{n\to\infty} F_{x^{(1)}y}^{(n)}(t_0^1)=\frac{1-p}{p}.
\end{eqnarray}\

Next, consider $t_0^2\equiv \frac{8}{9}=3^0-3^{-2}$ and $t_1^2\equiv\frac{2}{9}=3^{-1}-3^{-2}$. As in the previous case, we can construct sequences $\{M_i\}_{i=1}^\infty$ and $\{N_i\}_{i=1}^\infty$ such that for $x^{(2)}\equiv r_10^{M_1}2^{M_2}0^{M_3}2^{M_4}\ldots$, we have
\begin{equation*}
\left|F_{x^{(2)}y}^{(N_j)}(t_1^2)-1\right|<\frac{1}{j}
\end{equation*}
for even $j$, and we have
\begin{eqnarray*}
\left|F_{x^{(2)}y}^{(N_j)}(t_0^2)-\frac{1-p}{p}\right|&<&\frac{1}{j}\\
\left|F_{x^{(2)}y}^{(N_j)}(t_1^2)-\left(\frac{1-p}{p}\right)^2\right|&<&\frac{1}{j}
\end{eqnarray*}
for odd $j$. Therefore,
\begin{eqnarray*}
F_{x^{(2)}y}^\star(t_1^2)&=&1,\\
F_{x^{(2)}y}(t_0^2)&=&\frac{1-p}{p},\\
F_{x^{(2)}y}(t_1^2)&=&\left(\frac{1-p}{p}\right)^2.
\end{eqnarray*}
Given that $F_{xy}^\star$ and $\dol$ are both non-decreasing, $F_{x^{(2)}y}^\star(t)=1$ for every $t>t_1^2$, and $F_{x^{(2)}y}(t)\in \left[\left(\frac{1-p}{p}\right)^2,\frac{1-p}{p}\right]$ for $t\in [t_1^2,t_0^2]$. It follows that the area between $F_{x^{(2)}y}$ and $F_{x^{(2)}y}^\star$ is at least 
\begin{equation}
\int_0^1 F_{x^{(2)}y}^\star(t)-F_{x^{(2)}y}(t)\dt \ge (t_0^2-t_1^2)\cdot\left(1-\frac{1-p}{p}\right)
\end{equation}
The `worst case' is $F_{x^{2}y}(t)=\frac{1-p}{p}$ for $t\in (t_1^2,t_0^2]$(see Fig.~\ref{dc2}).
\begin{figure}
\begin{center}
\includegraphics[height=8cm,width=8cm]{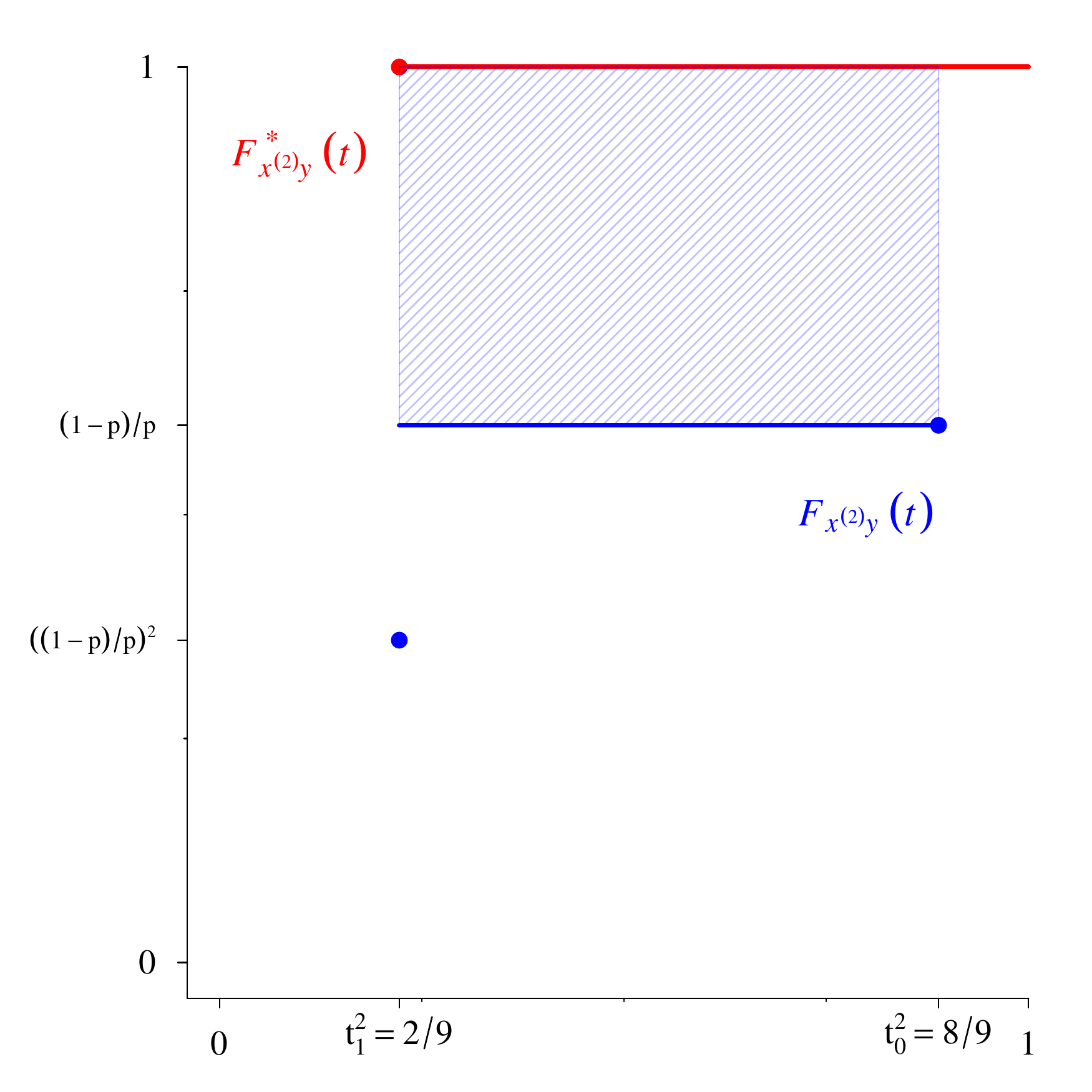}
\end{center}
\caption{Parts of the functions $F_{x^{(2)}y}$ and $F_{x^{(2)}y}^\star$ (the worst possible case).}
\label{dc2}
\end{figure}\

Similarly, for a positive integer $k$ and $t_l^k=3^{-l}-3^{-k}$, $l=0,1,\ldots,k-1$ we can construct $x^{(k)}$ such that
\begin{eqnarray*}
F_{x^{(k)}y}^\star(t_{k-1}^k)&=&1,\\
F_{x^{(k)}y}(t_l^k)&=&\left(\frac{1-p}{p}\right)^{l+1}
\end{eqnarray*}
for $l=0,1,\ldots,k-1$. As in the previous case, $F_{x^{(k)}y}^\star(t)=1$ for every $t>t_{k-1}^k$, and 
\begin{equation}
F_{x^{(k)}y}(t)\in \left[\left(\frac{1-p}{p}\right)^{i+1},\left(\frac{1-p}{p}\right)^i\right]
\end{equation}
for $t\in[t_i^k,t_{i-1}^k]$, $i=1,2,\ldots,k-1$. 
\begin{figure}
\begin{center}
\includegraphics[height=8cm,width=8cm]{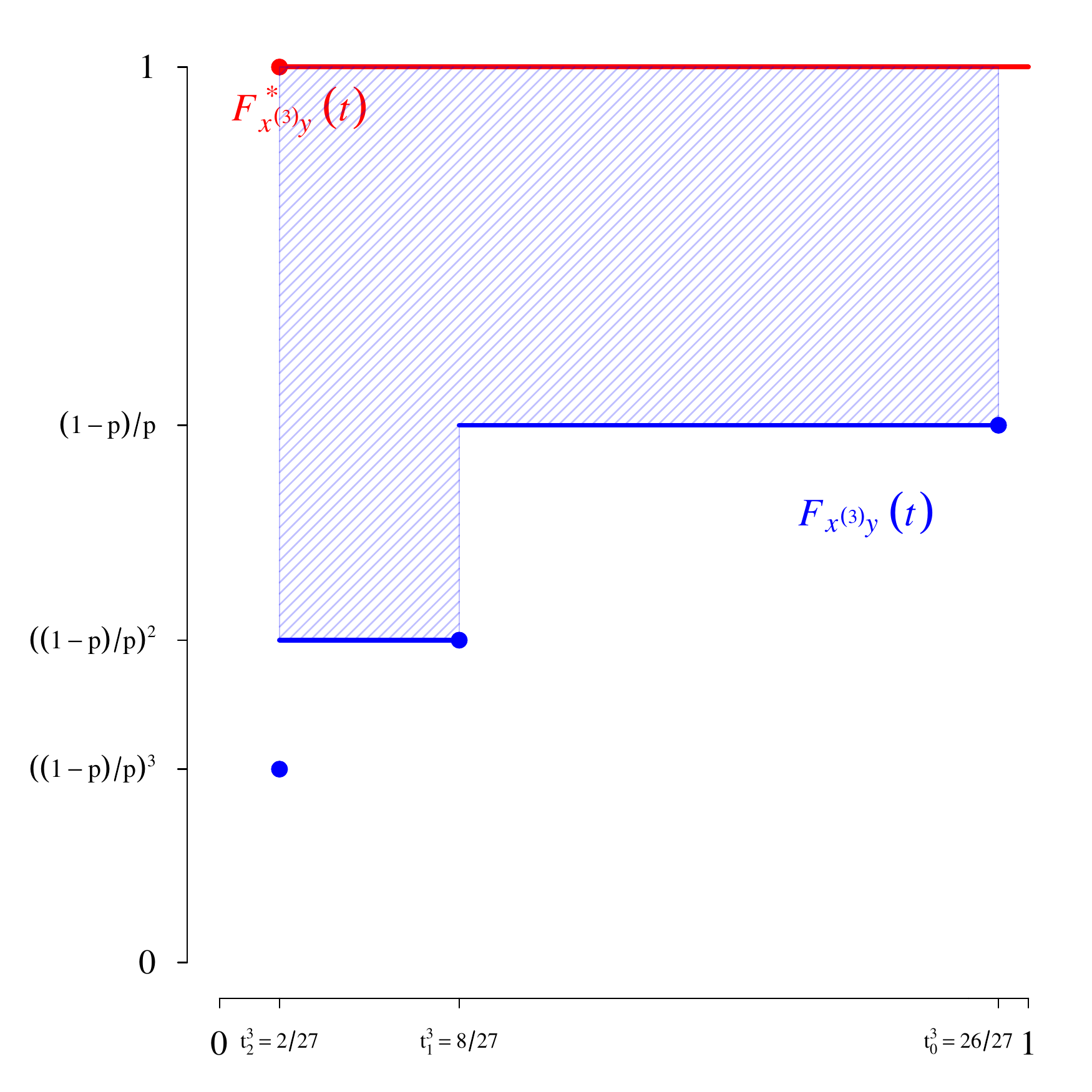}
\end{center}
\caption{Parts of the functions $F_{x^{(3)}y}$ and $F_{x^{(3)}y}^\star$ (the worst possible case).}
\end{figure}
Thus, the area between $F_{x^{(k)}y}$ and $F_{x^{(k)}y}^\star$ is at least
\begin{equation*}
\begin{aligned}
\int_0^1 & F_{x^{(k)}y}^\star(t)-F_{x^{(k)}y}(t)\dt \ge \\
&\ge \sum_{i=1}^{k-1}(t_{i-1}^k-t_i^k)\cdot\left(1-\left(\frac{1-p}{p}\right)^i\right)=
\sum_{i=1}^{k-1}(t_{i-1}^k-t_i^k) - 
\sum_{i=1}^{k-1}(t_{i-1}^k-t_i^k)\cdot \left(\frac{1-p}{p}\right)^i =\\
&=
t_0^k-t_{k-1}^k - \sum_{i=1}^{k-1}(3^{-(i-1)}-3^{-k}-(3^{-i}-3^{-k}))
\cdot \left(\frac{1-p}{p}\right)^i = \\
&=(3^0-3^{-k})-(3^{-(k-1)}-3^{-k}) - 
\sum_{i=1}^{k-1}(3^{-(i-1)}-3^{-i})\cdot \left(\frac{1-p}{p}\right)^i =\\
&= 1-3^{-(k-1)}-\sum_{i=1}^{k-1}(3^{-i+1}-3^{-i})\left(\frac{1-p}{p}\right)^i.
\end{aligned}
\end{equation*}
For $k\to\infty$, we get the same summation as in the proof of Proposition~\ref{dva}; hence,
\begin{equation}
\liminf_{k\to\infty}\int_0^1 F_{x^{(k)}y}^\star(t)-F_{x^{(k)}y}(t)\dt \ge \frac{6p-3}{4p-1}.
\end{equation}
However, from Proposition~\ref{dva}, we have 
\begin{equation}
\int_0^1 F_{x^{(k)}y}^\star(t)-F_{x^{(k)}y}(t)\dt \le \frac{6p-3}{4p-1}
\end{equation}
for every $k$; therefore,
\begin{equation}
\lim_{k\to\infty}\int_0^1 F_{x^{(k)}y}^\star(t)-F_{x^{(k)}y}(t)\dt = \frac{6p-3}{4p-1}.
\end{equation}
This concludes the proof.
\end{proof}
Now we will prove Lemmas~\ref{pat} and~\ref{sest}.
\begin{proof}[Proof of Lemma~\ref{pat}]
First, consider a number $w'$ of the form $r0^\infty$ and a set 
\begin{equation}\label{Cn}
C_n^{\ell(r)}\equiv \left\{\omega\in\Omega: \sum_{i=1}^m(I_f(\omega_i)-I_g(\omega_i))=\ell(r)\text{ for some } m\in\{1,2,\ldots n\} \right\}
\end{equation}
(recall that $\ell(r)$ denotes the length of the sequence $r$). Let $\omega\in C_n^{\ell(r)}$ and denote 
\begin{equation}
m_0\equiv\min\left\{ m\in\{1,2,\ldots\}: \sum_{i=1}^m (I_f(\omega_i)-I_g(\omega_i))=\ell(r) \right\}.
\end{equation}
Then, $w_{m_0}'(\omega)=0^\infty$ (because $f\circ g=id$ and $f$ acts as a shift) and $w_n'(\omega)=0^\infty$. Therefore, 
\begin{equation*}
P(|w_n^\prime-y_n|<3^{-k})=P(w_n'<3^{-k})\ge P(C_n^{\ell(r)}),
\end{equation*}
where $k$ is the smallest integer for which $3^{-k}<t$.
However, $\displaystyle P(C_n^{\ell(r)}) \to 1$ as $n\to\infty$ (by Lemma~\ref{rw}). Hence, $P(|w'_n-y_n|<3^{-k})\to 1$ and $F_{w'y}^{(n)}(3^{-k})\to 1$. Consequently, there exists an arbitrarily large positive integer $N$ such that 
\begin{equation}
F_{w'y}^{(N)}(3^{-k})=\frac{1}{N}\sum_{i=0}^{N-1}P(w'_i<3^{-k})>1-\varepsilon.
\end{equation}
However, the events $(w'_i<3^{-k}), i=0,1,\ldots,N-1$ are only affected by the first $N+k$ terms of the ternary representation of $w'$. Hence, the remaining terms can be replaced by the sequence $s$. It follows that for a number $w\equiv r0^{N+k-\ell(r)}s$, we have 
\begin{equation}
F_{wy}^{(N)}(t)\ge F_{wy}^{(N)}(3^{-k})>1-\varepsilon.
\end{equation}
Therefore,
\begin{equation}
|F_{wy}^{(N)}(t)-1|<\varepsilon.
\end{equation}
\end{proof}
\begin{proof}[Proof of Lemma~\ref{sest}]
First, consider a number $z'$ of the form $r2^\infty$. We will show that 
\begin{equation}\label{eq1}
\lim_{n\to\infty} P(z'_n<3^{-l}-3^{-k}) = \left(\frac{1-p}{p}\right)^{l+1}.
\end{equation}
Let $\displaystyle \omega\in B_n^{l+1}$, where $\displaystyle B_n^{l+1}$ was defined in~(\ref{Bn}). Then, $z'_n(\omega)$ begins with $l+1$ zeros. Hence, 
\begin{equation}
z'_n(\omega)\le 3^{-(l+1)} < 3^{-(l+1)}(3-3^{l+1-k})=3^{-l}-3^{-k}.
\end{equation}
Therefore, 
\begin{equation}
P(z'_n<3^{-l}-3^{-k})\ge P(B_n^{l+1})\to \left(\frac{1-p}{p}\right)^{l+1}.
\end{equation}
Now let $\displaystyle \omega\notin B_n^{l+1}$ and $\omega\in C_n^{\ell(r)}$ at the same time ($C_n^\ell(r)$ was defined in~(\ref{Cn})). Then $z'_n(\omega)$ must be from the set $\displaystyle \{2^\infty, 02^\infty,0^22^\infty,\ldots,0^l2^\infty\}$, therefore $z'_n(\omega)\ge 3^{-l}\ge 3^{-l}-3^{-k}$. It follows that 
\begin{equation}
P\left((B_n^{l+1})^C\cap C_n^{\ell(r)}\right)\le P(z'_n\ge 3^{-l}-3^{-k}). 
\end{equation}
Hence, 
\begin{equation*}
\begin{aligned}
P(z_n<3^{-l}-3^{-k})&\le P\left(\left((B_n^{l+1})^C\cap C_n^{\ell(r)}\right)^C\right)=P\left(B_n^{l+1}\cup (C_n^{\ell(r)})^C\right)\le \\
&\le P(B_n^{l+1})+P\left((C_n^{\ell(r)})^C\right)\to \left(\frac{1-p}{p}\right)^{l+1}, 
\end{aligned}
\end{equation*}
because $\displaystyle P(C_n^{\ell(r)})\to 1$; therefore, the equality in~(\ref{eq1}) holds. Consequently, 
\begin{equation}
F_{z'y}^{(n)}(3^{-l}-3^{-k})\to \left(\frac{1-p}{p}\right)^{l+1}, 
\end{equation}
so there exists arbitrarily large $N$ such that 
\begin{equation}
\left(\frac{1-p}{p}\right)^{l+1} -\varepsilon < F_{z'y}^{(N)}(3^{-l}-3^{-k})=\frac{1}{N}\sum_{i=0}^{N-1}P(z'_i<3^{-l}-3^{-k}) < \left(\frac{1-p}{p}\right)^{l+1}+\varepsilon.
\end{equation}
However, as in the proof of the previous Lemma, events $(z'_i<3^{-l}-3^{-k}), i=0,\ldots N-1$ are affected by only the first $N+k$ terms of the ternary representation of $z'$. Hence, the remaining terms can be replaced by the sequence $s$. Consequently, for $z$ of the form $r2^{N+k-\ell(r)}s$, we have 
\begin{equation}
\left|F_{zy}^{(N)}(3^{-l}-3^{-k})-\left(\frac{1-p}{p}\right)^{l+1}\right|<\varepsilon
\end{equation}
for every $l\in\{0,1,\ldots,k-1\}$.
\end{proof}
\end{example}

\begin{example}\label{priklad}
Let $f,g:[0,1]\to[0,1]$ be defined by 
\begin{equation}
f(x)=
\begin{cases}
3x & \text{ if } x\in \left[0,\frac{1}{3}\right],\\
-x+\frac{4}{3} & \text{ if } x\in\left(\frac{1}{3},\frac{2}{3}\right),\\
x & \text{ if } x\in \left[\frac{2}{3},1\right]
\end{cases}
\end{equation}
and
\begin{equation}
g(x)=
\begin{cases}
x & \text{ if } x\in \left[0,\frac{1}{3}\right],\\
-x+\frac{3}{2} & \text{ if } x\in\left(\frac{1}{3},\frac{2}{3}\right),\\
3x-2 & \text{ if } x\in \left[\frac{2}{3},1\right].
\end{cases}
\end{equation}
\begin{figure}
\begin{center}
\includegraphics[height=8cm,width=8cm]{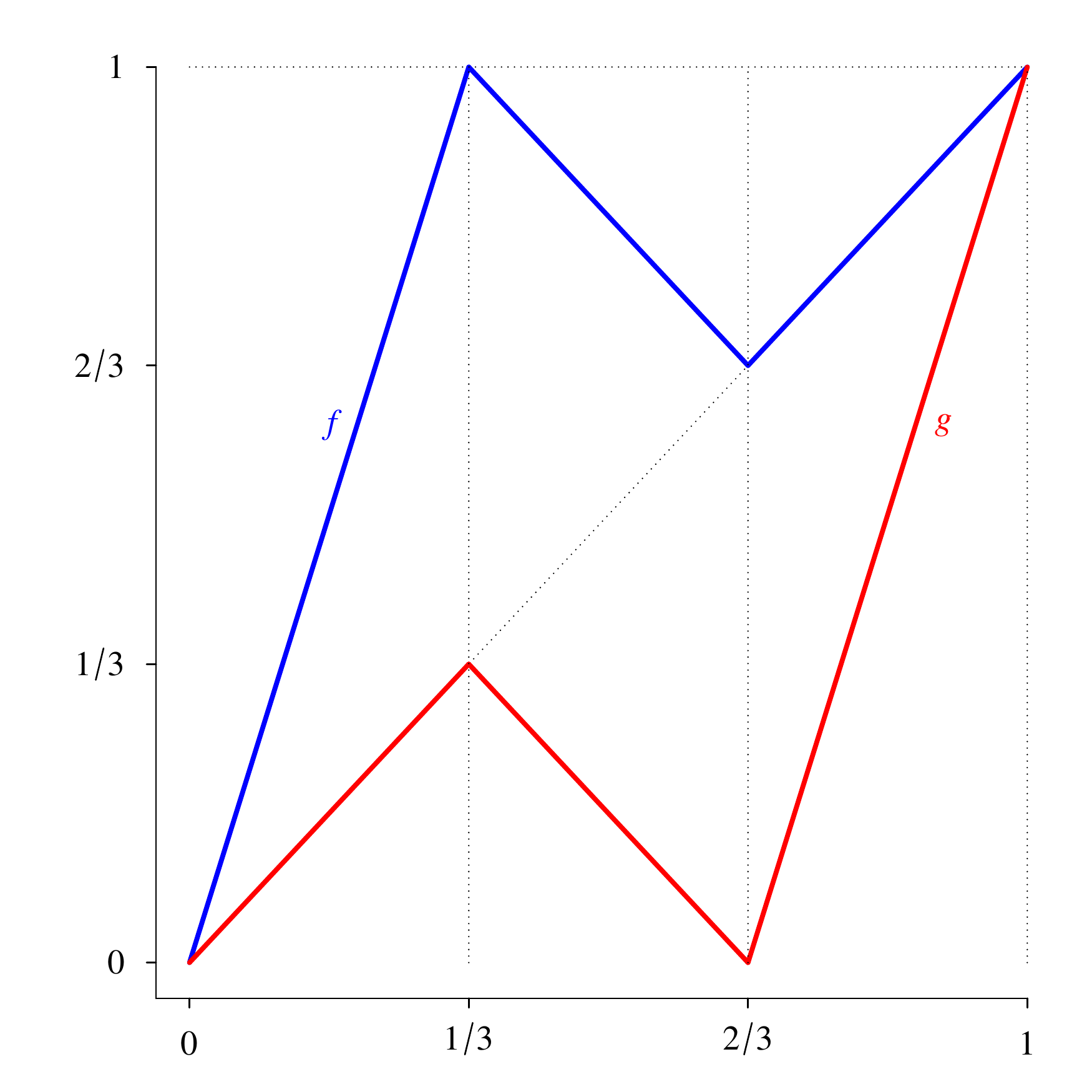}
\end{center}
\caption{The functions $f$ and $g$ in Example~\ref{priklad}}
\end{figure}
It can be seen that in the ternary representation, we have $f(0s)=s, f(1s)=2\overline{s}, f(2s)=2s$, $g(0s)=0s, g(1s)=0\overline{s},$ and $g(2s)=s$. Using similar techniques to the previous example, for any positive integer $k$, there exists a sequence $\{M_n\}_{n=1}^\infty$ of positive integers such that for $x=0^{M_1}2^{M_2}0^{M_3}2^{M_4}\ldots$ and $y=0$, we have $\dol(1-3^{-k})=0$ and $\hor(3^{-k})=1$. Consequently, $\mu(f,g)=1$ for every $p\in(0,1)$. However, both $f$ and $g$ are clearly nonchaotic (every trajectory converges to a fixed point). Therefore $\mu(f,g)=0$ for $p\in\{0,1\}$.
\end{example}

\section{Instability}
For $f,g$ in Example~\ref{priklad}, the measure $\mu(f,g,p):[0,1]\to[0,1]$ is not continuous in $p=0$ and $p=1$. The following theorem~\ref{unst} shows that for any $p\in(0,1)$ the function $\mu(f,g):C(I,I)\times C(I,I)\to [0,1]$ is also discontinuous at any point $(f,g)$. 
%In the Example~\ref{priklad}, we could see that the measure $\mu$ is discontinuous as a function of the probability $p$. In the following theorem, we will show that it is also discontinuous as a function of a pair of continuous functions $(f,g)$.

\begin{theorem}\label{unst}
Let $f,g:I\to I$ be continuous, and let $p$ be from the interval $(0,1)$. Then, for any $\varepsilon>0$, there exist continuous functions $\displaystyle f^\star$ and $\displaystyle g^\star$ such that $\displaystyle d(f,f^\star)<\varepsilon$, $\displaystyle d(g,g^\star)<\varepsilon$, and $\displaystyle \mu(f^\star,g^\star)=0$. The metric $d$ is given by
\begin{equation}
d(f,f^\star)\equiv \sup_{x\in I}|f(x)-f^\star(x)|.
\end{equation}
\end{theorem}
\begin{proof}
For the sake of simplicity, we will only prove this theorem for the interval $I=[0,1]$.

As $f$ and $g$ are continuous on the compact set, they are also uniformly continuous. Consequently, for any $\varepsilon$ there exists $\delta>0$ such that 
\begin{equation}
|x-y|<\delta \Rightarrow (|f(x)-f(y)|<\varepsilon \wedge |g(x)-g(y)|<\varepsilon)
\end{equation}
for every $x,y\in[0,1]$. Let $n$ be a positive integer such that $\frac{1}{n}<\delta$. The idea is to construct $\displaystyle f^\star, g^\star$ such that for the set
\begin{equation}
A\equiv\left\{f\left(\frac{0}{n}\right),f\left(\frac{1}{n}\right),\ldots, f\left(\frac{n}{n}\right),g\left(\frac{0}{n}\right),g\left(\frac{1}{n}\right),\ldots, g\left(\frac{n}{n}\right)\right\},
\end{equation}
we have
\begin{itemize}
\item $\displaystyle f^\star(A)\subseteq A$ and $\displaystyle g^\star(A)\subseteq A$,
\item $P(x_n\in A)\to 1$ for every $x\in[0,1]$ (in the random dynamical system generated by the functions $\displaystyle f^\star$ and $\displaystyle g^\star$).
\end{itemize}
By Theorem~\ref{Finite}, these conditions ensure that $\displaystyle\mu(f^\star,g^\star)=0$ because $A$ is a finite set. 

Consider an interval $I_k\equiv \left[\frac{k}{n},\frac{k+1}{n}\right], k=0,1,\ldots,n-1$. We will construct $\displaystyle f^\star$ and $\displaystyle g^\star$ on this interval in the following way. 
\begin{enumerate}
\item If $A\cap int(I_k)=\emptyset$, then we set
\begin{eqnarray*}
f^\star\left(\frac{k}{n}\right)&\equiv& f\left(\frac{k}{n}\right),\\
f^\star\left(\frac{2k+1}{2n}\right)&\equiv& f\left(\frac{k}{n}\right),\\
f^\star\left(\frac{k+1}{n}\right)&\equiv& f\left(\frac{k+1}{n}\right),\\
g^\star\left(\frac{k}{n}\right)&\equiv& g\left(\frac{k}{n}\right),\\
g^\star\left(\frac{2k+1}{2n}\right)&\equiv& g\left(\frac{k+1}{n}\right),\\
g^\star\left(\frac{k+1}{n}\right)&\equiv& g\left(\frac{k+1}{n}\right)
\end{eqnarray*}
and $\displaystyle f^\star, g^\star$ are linear between these points. 
\begin{figure}
\begin{center}
\includegraphics[height=8cm,width=8cm]{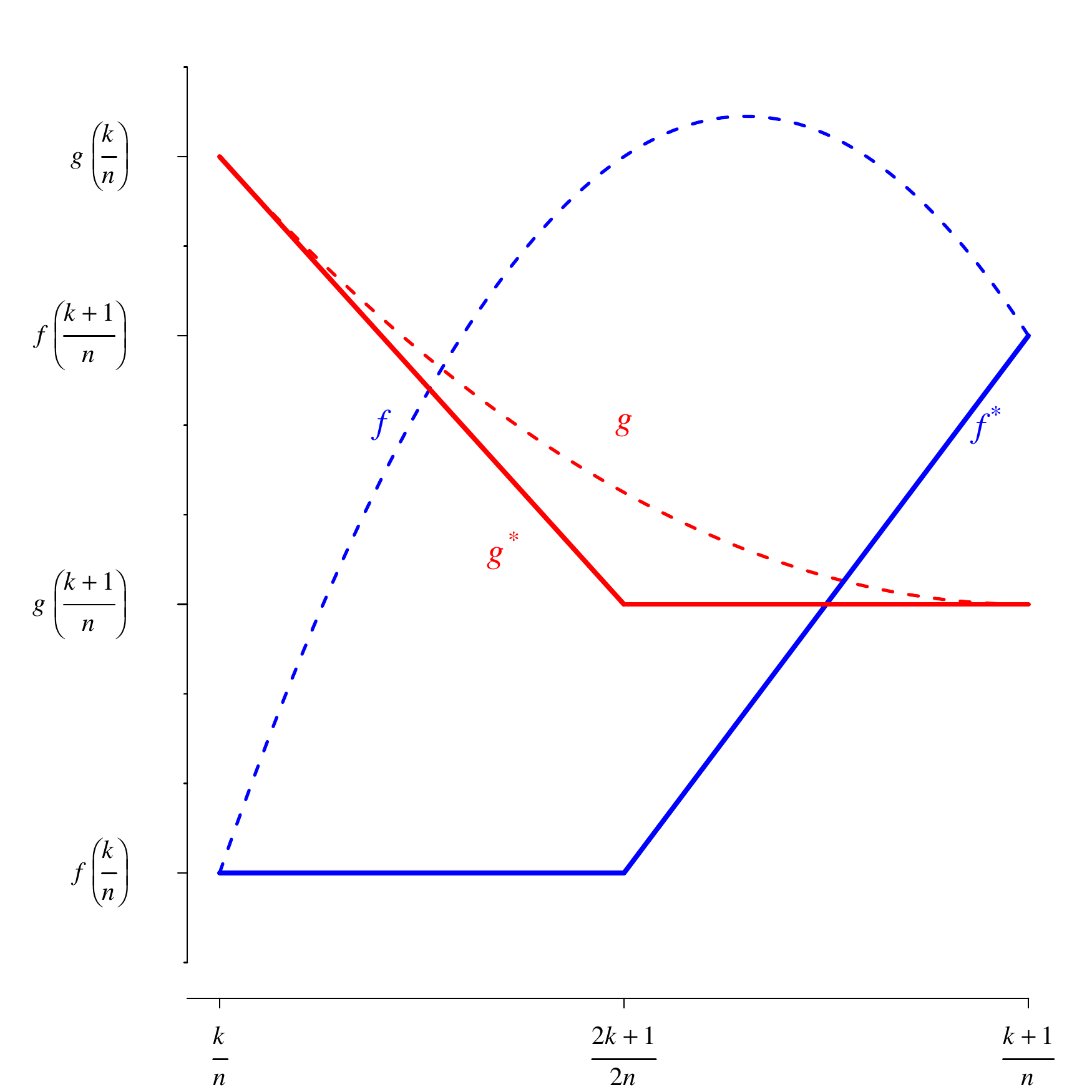}
\end{center}
\caption{The functions $\displaystyle f^\star$ and $\displaystyle g^\star$ when $A\cap int(I_k)=\emptyset$.}
\end{figure}
\item If $A\cap int(I_k)=\{a_1<a_2<\ldots<a_m\}$, then set $b_0\equiv \frac{\frac{k}{n}+a_1}{2}$ (the midpoint of $\frac{k}{n}$ and $a_1$) and $b_1\equiv\frac{a_m+\frac{k+1}{n}}{2}$ (the midpoint of $a_m$ and $\frac{k+1}{m}$). The functions $\displaystyle f^\star$ and $\displaystyle g^\star$ are given by 
\begin{eqnarray*}
f^\star\left(\frac{k}{n}\right)&\equiv& f\left(\frac{k}{n}\right),\\
f^\star\left(b_1\right)&\equiv& f\left(\frac{k}{n}\right),\\
f^\star\left(\frac{k+1}{n}\right)&\equiv& f\left(\frac{k+1}{n}\right),\\
g^\star\left(\frac{k}{n}\right)&\equiv& g\left(\frac{k}{n}\right),\\
g^\star\left(b_0\right)&\equiv& g\left(\frac{k+1}{n}\right),\\
g^\star\left(\frac{k+1}{n}\right)&\equiv& g\left(\frac{k+1}{n}\right),
\end{eqnarray*}
and $\displaystyle f^\star, g^\star$ are linear between these points. 
\begin{figure}
\begin{center}
\includegraphics[height=8cm,width=8cm]{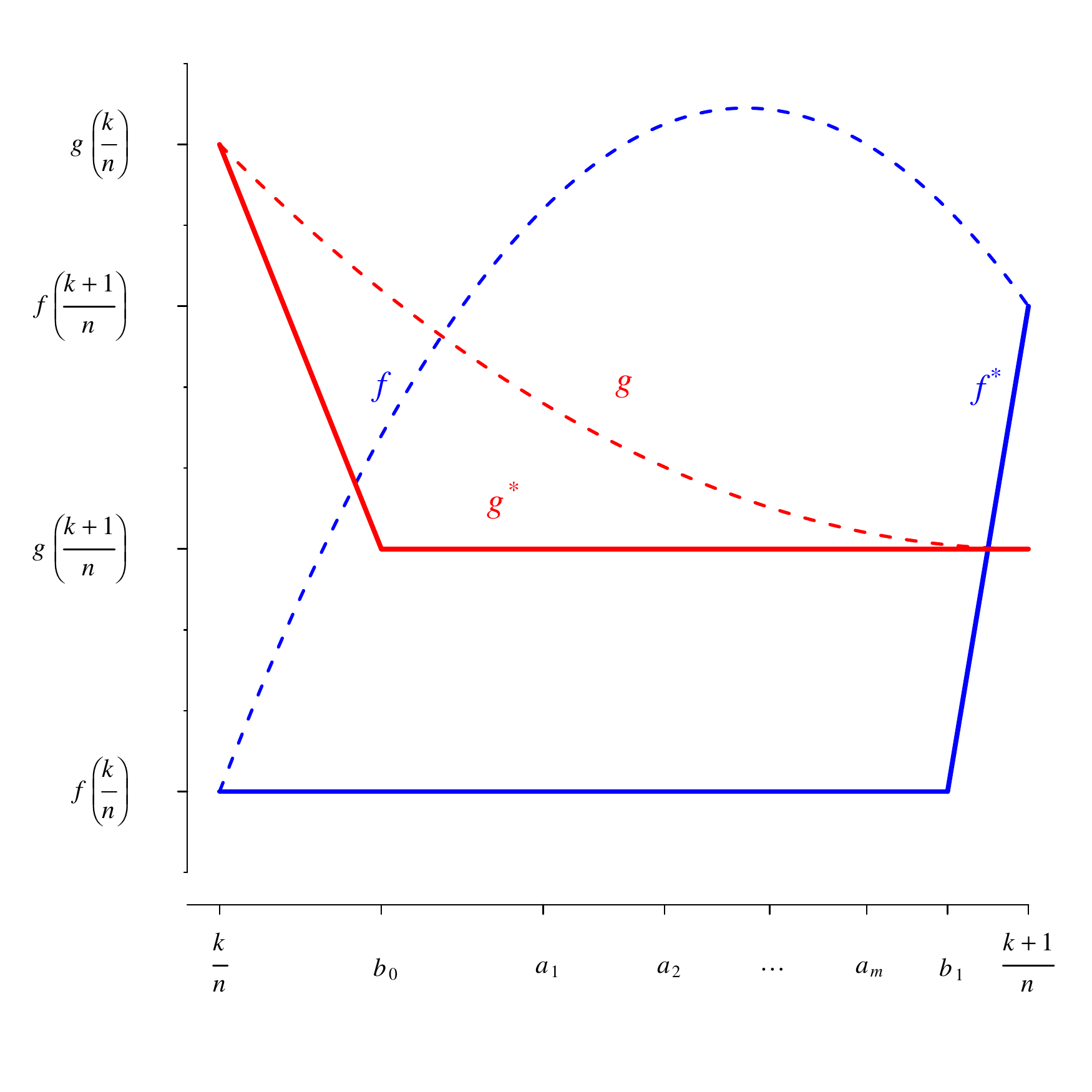}
\end{center}
\caption{The functions $\displaystyle f^\star$ and $\displaystyle g^\star$ when $A\cap int(I_k)\neq\emptyset$.}
\end{figure}
\end{enumerate}
Clearly $\displaystyle d(f,f^\star)<\varepsilon$ because if $|f(x)-f(x^\star)|\ge\varepsilon$ for some $x\in I_k$, then $|f(x)-f(\frac{k}{n})|\ge\varepsilon$ or $|f(x)-f(\frac{k+1}{n})|$ (because $\displaystyle f^\star(x)$ is between the points $f(\frac{k}{n})$ and $f(\frac{k+1}{n})$ on the interval $I_k$). However, this is not possible because $|x-\frac{k}{n}|<\frac{1}{n}<\delta$ and $|x-\frac{k+1}{n}|<\frac{1}{n}<\delta$. Similarly, $\displaystyle d(g,g^\star)<\varepsilon$. 

Next, it can be seen that $\displaystyle f^\star(A)\subseteq A$, $\displaystyle g^\star(A)\subseteq A$, and 
\begin{equation}
P(x_{n}\notin A)\le (\max(p,1-p))^n\to 0. 
\end{equation}
Therefore, $P(x_n\in A)\to 1$ for every $x\in[0,1]$. Consequently, $\displaystyle\mu(f^\star,g^\star)=0$.
\end{proof}

This theorem implies that every system (even chaotic ones) can be modified by an arbitrarily small change to a nonchaotic system. Interestingly, this does not hold in deterministic systems. If a continuous map $f$ defined on the closed interval $I$ is distributionally chaotic, then there exists $\varepsilon>0$ such that for every continuous map $\displaystyle f^\star$, we have
\begin{equation}
d(f,f^\star)<\varepsilon \Rightarrow f^\star \text{ is distributionally chaotic}\Leftrightarrow \mu(f^\star)>0. 
\end{equation}
This follows from the equivalency of distributional chaos and the positive topological entropy on the closed interval \cite[see][]{Smital1} and from the fact that the topological entropy is lower semi-continuous \cite[e.g.][]{Block}. 

This enables us to construct a random dynamical system generated by distributionally chaotic functions $\displaystyle f^\star$ and $g^\star$, which is not distributionally chaotic. Consider two distributionally chaotic functions $f$ and $g$ and $\varepsilon>0$ such that $\displaystyle d(f,f^\star)<\varepsilon$ and $\displaystyle d(g,g^\star)<\varepsilon$ imply $\displaystyle \mu(f^\star)>0$ and $\displaystyle \mu(g^\star)>0$. However, for this $\varepsilon$, there are $\displaystyle f^\star$ and $\displaystyle g^\star$ such that the random dynamical system~(\ref{system0}) is not distributionally chaotic. In this case, randomness helps us to `remove' the chaos from the system in some way. 

\section{Concluding remarks}
\begin{enumerate}
\item In Section 3, we defined distributional chaos for a random dynamical system generated by two maps. It is possible to extend this definition to systems generated by arbitrarily, and even uncountably, many maps.

\item In Section 6, we showed that two distributionally chaotic functions can generate a random dynamical system that is not distributionally chaotic. It can be shown that a nonchaotic system can even arise from the two mixing maps. Consider the functions 
\begin{center}
$
\begin{matrix}
\displaystyle
f(x)=
\begin{cases} 
4x & \text{if } x\in\left[0,\frac{1}{4}\right],  \\ 
-\frac{1}{8}x+\frac{33}{32} & \text{if } x\in\left(\frac{1}{4},\frac{1}{2}\right],\\
\frac{1}{8}x+\frac{29}{32} & \text{if } x\in\left(\frac{1}{2},\frac{3}{4}\right],\\
-4x+4  & \text{if } x\in\left(\frac{3}{4},1\right]
\end{cases}
& 
\text{and} & 
g(x)=
\begin{cases}
\frac{1}{8}x+\frac{15}{32} & \text{if } x\in\left[0,\frac{1}{4}\right],  \\ 
4x-\frac{1}{2} & \text{if } x\in\left(\frac{1}{4},\frac{3}{8}\right],\\
-4x+\frac{5}{2} & \text{if } x\in\left(\frac{3}{8},\frac{5}{8}\right],\\
4x+\frac{5}{2}  & \text{if } x\in\left(\frac{5}{8},\frac{3}{4}\right],\\
\frac{1}{8}x+\frac{13}{32}  & \text{if } x\in\left(\frac{3}{4},1\right]
\end{cases}
\end{matrix}
$
\end{center} 
and the system 
\begin{equation}
x_{n+1}=
\begin{cases}
f(x_n) & \text{with probability } \frac{1}{2},\\
g(x_n) & \text{with probability } \frac{1}{2}.
\end{cases}
\end{equation}
\begin{figure}
\begin{center}
\includegraphics[height=8cm,width=8cm]{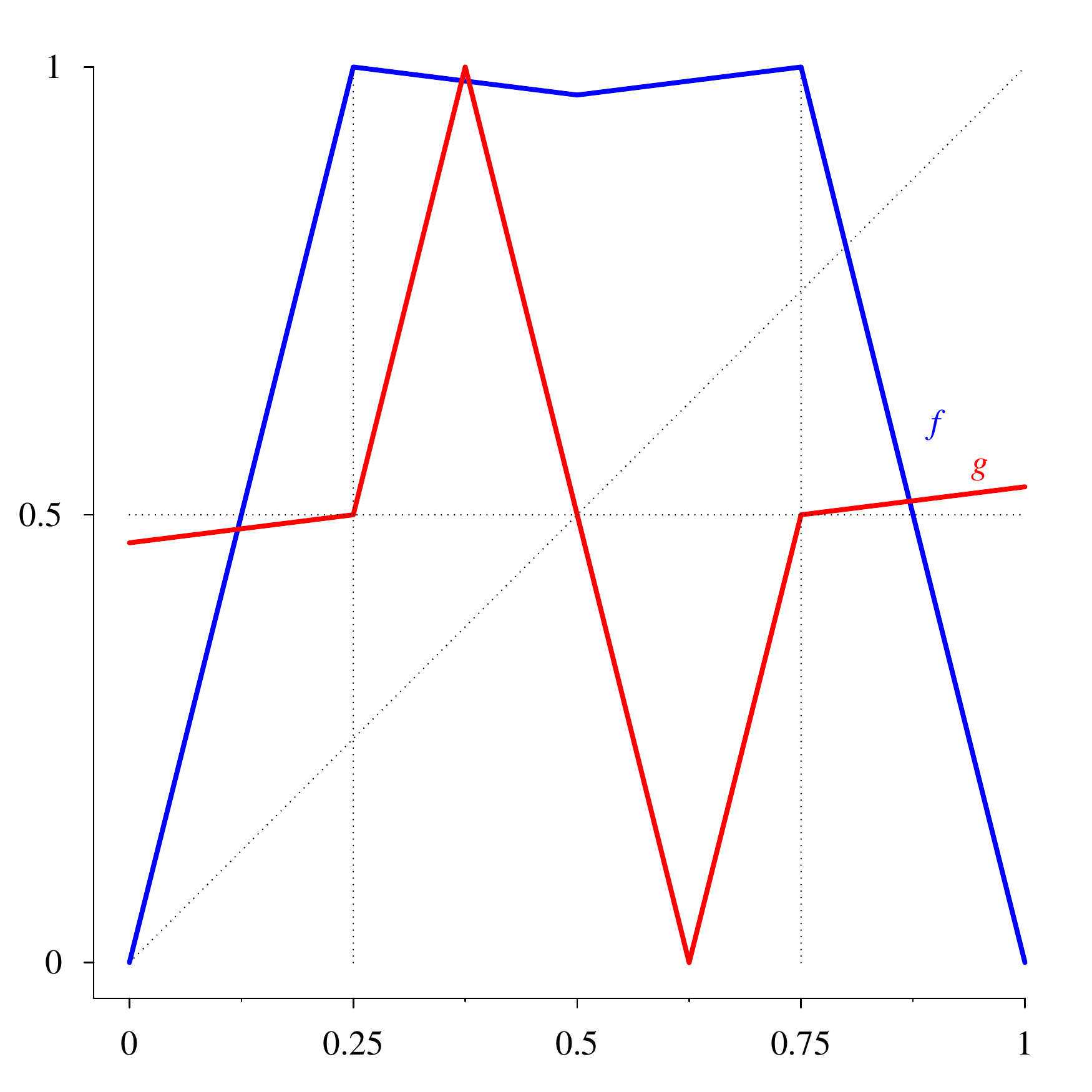}
\end{center}
\caption{The functions $f$ and $g$.}
\end{figure}
In this case, $\mu(f,g)=0$. The exact proof is quite complicated, but the concept behind it is very simple - if $x_n$ and $y_n$ are in the same quarter of the interval $[0,1]$ (which will happen infinitely many times), then either $|x_{n+1}-y_{n+1}|\le \frac{1}{8}|x_n-y_n|$ (with probability $\frac{1}{2}$) or $|x_{n+1}-y_{n+1}|\le 4|x_n-y_n|$ (with probability $\frac{1}{2}$). Therefore, the `average distance' between $x_n$ and $y_n$ will tend to $0$ almost everywhere; hence, $\mu(f,g)=0$.
\item Theorem~\ref{unst} says that the pairs of continuous functions $(f,g)$, which generate nonchaotic systems, are dense in the space $C(I,I)\times C(I,I)$. We conjecture that this is also true for the pairs that generate chaotic systems, i.e., for every pair $(f,g)$ there is a `chaotic pair' $\displaystyle(f^\star,g^\star)$, which is arbitrarily close to $(f,g)$. In some cases, it is easy to construct such functions, but we have not yet been able to find a universal algorithm.
\end{enumerate}

\section*{Funding}
This work was supported by the Slovak Scientific Grant Agency under VEGA Grant No. 2/0054/18 (both authors); Comenius University in Bratislava under Grant UK/151/2018 (the first author).

\end{document}